\documentclass{article}

\usepackage[utf8]{inputenc}
\usepackage[T1]{fontenc}
\usepackage[colorlinks = true, linkcolor = blue, urlcolor = blue, citecolor = blue, anchorcolor = blue]{hyperref} 
\usepackage{xcolor}
\usepackage{amsmath,amsthm,amssymb}
\usepackage{graphicx}
\usepackage{subfigure}
\usepackage{multicol}
\usepackage{listings}
\usepackage[round]{natbib}
\usepackage{authblk}

\newtheorem{theorem}{Theorem}[section]
\newtheorem{definition}{Definition}[section]
\theoremstyle{definition}
\newtheorem{example}{Example}

\newtheorem{proposition}{Proposition}

\def\Pr{\mathrm{P}}
\def\E{\mathrm{E}}
\def\bm#1{\mathbf{#1}}

\title{Pareto processes for threshold exceedances in spatial extremes}
\date{}
\author[1]{Clément Dombry}
\author[2]{Juliette Legrand}
\author[3]{Thomas Opitz}

\affil[1]{{\small Université de Franche Comté, CNRS, LmB (UMR 6623), F-25000 Besançon, France}}
\affil[2]{{\small Univ Brest, CNRS UMR 6205, Laboratoire de Mathématiques de Bretagne Atlantique, France}}
\affil[3]{{\small Biostatistics and Spatial Processes, INRAE, Avignon, France}}

\begin{document}
\maketitle

This preprint is an author version of a chapter to appear in a collaborative book.

\tableofcontents
\pagebreak

\section{Introduction}

Modeling extremes of spatial and spatio-temporal phenomena is an important branch of extreme value theory that has seen substantial developments in the last decades and found many applications, especially in environmental sciences \citep{DG12,DPR12,Huser2022}. An example is the implementation of offshore wind farms that can be vulnerable to extreme sea and wind conditions. In this specific context, application studies have focused on modeling the extremes of significant wave heights, a quantity measuring the severity of a sea state, over a given area (see Figure \ref{fig::MapHsData}). Statistical extreme-value analyses for these data will be used to illustrate the models and methods reviewed in this chapter.

\begin{figure}[htb]\label{fig::MapHsData}
\begin{center}
\includegraphics[scale=0.35]{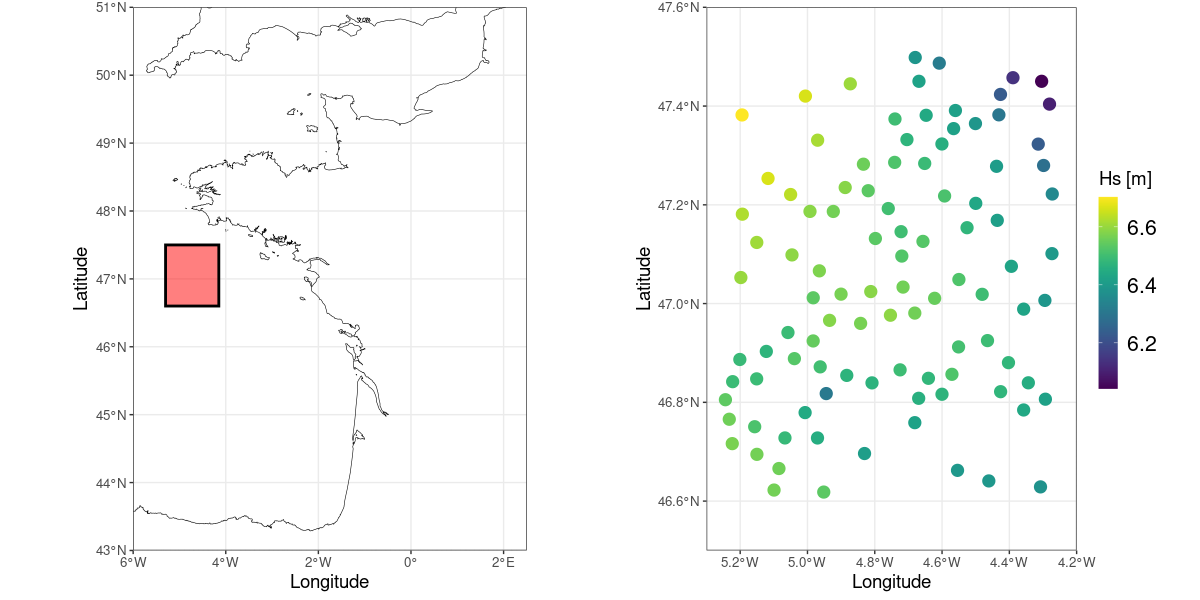}
\end{center}
\caption[Map data]{Daily hindcast\footnotemark\ significant wave heights ($H_s$, [m]) off the French coast during winter months (JFM) from 1995 to 2015. Left panel: study region. Right panel: marginal $95\%$ empirical quantile at each location.}
\end{figure}

Extreme value \footnotetext{Hindcast data refers to historical outputs from numerical models} statistics rely on two important different paradigms, the so-called Block-Maximum (BM) and Peaks-Over-Threshold (POT) methods. In a nutshell, the BM approach is based on the study of  the maxima of the phenomenon  over different periods of times (typically monthly or yearly maxima), while the POT approach focuses on occurrences of the phenomenon that exceed a fixed high threshold. Probability theory then shows that two classes of limit distributions arise when considering asymptotic behavior. By the Fisher-Tippett-Gnedenko theorem, the probability distribution of maxima over large blocks is well approximated by a Generalized Extreme Value (GEV) distribution, whereas the Pickands-Balkema-de Haan theorem states that the probability distribution of exceedances over a high threshold is well approximated by a Generalized Pareto (GP) distribution. Importantly, the two theorems hold under the same tail-regularity assumption on the distribution of the underlying phenomena (the so-called first-order condition), so that the choice between the two approaches belongs to the statistician and could be motivated by the application setting.  Note in passing that a third approach, the so-called Point-Process approach, has the conceptual and theoretical advantage to encompass both the BM and POT methods. However, it is  far less popular in practice and in applications, so that we do not 
emphasize it here.

In the univariate setting, the two notions of maxima and of exceedances over a threshold are unambiguous. This is not any more the case in the multivariate or spatial setting. Regarding the BM approach, maxima are typically defined as componentwise maxima in the multivariate setting, and as pointwise maxima in the spatial setting. This gives rise to the theory of multivariate GEV distributions 
and of max-stable random fields.
One difficulty is that the use of componentwise and pointwise maxima results in ``composite observations" or ``composite events", since for instance the yearly maxima over a domain occur on different days at different locations. From a theoretical point of view, this is related to de Haan's constructive representation of a max-stable random vector (or random field) as the pointwise maximum of infinitely many random vectors (or random fields) representing the underlying extreme events \citep{deHaan1984}. The complexity of this representation makes the simulation and statistical analysis of max-stable random vectors and random fields quite difficult, and this has been  subject of very active and fruitful research throughout the last decades.

Regarding the POT approach in a multivariate or spatial setting, the meaning of an exceedance over a high threshold is not straightforward. It was first defined using the sup-norm, meaning that a random field is large if it is large at some location at least (in a multivariate setting, see \citet{RT2006}.
From a mathematical stance, this perspective allows to work in the space of (continuous) random fields endowed with the sup-norm and to develop the theory of simple Pareto processes \citep{FdH14}, where \emph{simple} refers to a specific choice of marginal distributions. Moreover, a close relationship between functional regular variation, simple max-stable processes and simple Pareto processes can be established. Interestingly, the simple Pareto process has a much simpler probabilistic structure than the corresponding max-stable process, which facilitates its simulation and statistical modeling. From an applied perspective, however, the use of the sup-norm is quite restrictive, since the notion of an extreme event can be very different depending on the application context. As proposed in \citet{DR15} and \citet{dFD18}, it makes sense to use a risk functional adapted to the specific application to decide whether an event is extreme or not, and to define functional exceedances of the spatial field over a high threshold with respect to this risk functional. The resulting models are called \emph{simple $r$-Pareto processes}, where $r$ stands for risk (note that \citet{DR15} uses the terminology $\ell$-Pareto process with $\ell$ standing for loss). The last conceptual step that ensures general practical applicability  through the notion of generalized Pareto processes is to allow for flexibility in the marginal tail index: while the simpler model considers a positive tail index that does not vary with the spatial location, \citet{dFD18} defines the \emph{generalized $r$-Pareto process} that allows for a general tail index (positive, null or negative) that can possibly vary spatially.

The structure of this chapter is as follows. In \S~\ref{sec:theory-pareto-processes}, we first review  the theory of Pareto processes and the various approaches developed in this context. Then, in \S~\ref{sec:Models-etc}, we discuss  modeling, inference and simulation of Pareto processes. Finally, the remaining two sections provide an application and some concluding remarks.

\section{Theory of Pareto processes}
\label{sec:theory-pareto-processes}
Following \citet{FdH14}, we first introduce the simple Pareto process and generalized Pareto process, which extend the POT approach to exceedances in a random field with respect to the sup-norm. Subsequently, exceedances with respect to a more general risk functional are considered for modeling flexibility and lead to the definition of generalized $r$-Pareto process \citep{DR15,dFD18}.

\subsection{The simple Pareto process}
We first define the  simple Pareto process and introduce its main properties. It is  motivated from the wish to adopt a POT perspective for spatial extremes, where the notion of exceedance is understood with respect to the sup-norm. 

Consider a compact subset $S$ of $\mathbb{R}^2$, which corresponds to the finite spatial window where the random process is observed, and let $C^+(S)$ denote the space of non-negative real-valued continuous functions on $S$ endowed with the sup-norm
\[
\|f\|=\max_{s\in S} |f(s)|,\quad f\in C^+(S).
\]
We say that  $f\in C^+(S)$ exceeds the threshold $u$ if $\|f\|>u$, meaning that $f(s)>u$ at some location $s\in S$.

\begin{definition}\label{ch16:def1}
  \index{simple Pareto process} 
  We say that $Z=(Z(s))_{s\in S}$ is a simple Pareto process if it can be written in the product form $Z=RY$ satisfying the following structure:
  \begin{itemize}
      \item $R$ has standard unit Pareto distribution, \emph{i.e.}, $\Pr(R>u)=1/u$ for $u>1$;
      \item $Y=(Y(s))_{s\in S}$ is a nonegative continuous random process satisfying $\Pr(\|Y\|=\theta)=1$ for some $\theta>0$, and $\E[Y(s)]=1$ for all $s\in S$;
      \item $R$ and $Y$ are independent.
  \end{itemize}
\end{definition}

The distribution of $Z$ depends only on the distribution of $Y$, called the spectral measure and denoted by $\rho$.  We call $Z$ the \emph{simple Pareto process with spectral measure $\rho$}. The spectral measure is defined on the sphere $C_\theta^+(S)=\{f\in C^+(S)\colon \|f\|=\theta\}$ of the sup-norm. The norm $\|Z\|$ 
satisfies
\[
\Pr(\|Z\|>u)=\Pr(\theta R>u)=\theta u^{-1},\quad u >\theta,
\]
and has unit Pareto distribution with scale $\theta$, which is called the areal extremal coefficient of $Z$. Note that $R$ and $Y$ can be obtained from $Z$ as $R=\theta^{-1}\|Z\|$ and $Y=\theta Z/\|Z\|$. Finally, the condition $\E[Y(s)]=1$ ensures a normalization of the marginal distributions: for all $s\in S$,
\[
\Pr(Z(s)>u)=\Pr(RY(s)>u)=\E[\min(1,u^{-1}Y(s))],\quad u>0, 
\]
is equal to $u^{-1}$ when $u>\theta$. Therefore, the marginal tail is ultimately standard unit Pareto.

\begin{example}\label{ex:construction-with-normalized-u} 
 Starting from any stochastic process $U=(U(s))_{s\in S}$ with paths almost surely in $C^+(S)\setminus\{0\}$, we can consider the normalized process $Y(s)=\tilde{Y}(s)/\mathrm{E}[\tilde{Y}(s)]$ with $\tilde{Y}=U/\|U\|$. 
 Then, the distribution of $Y$ defines a valid spectral measure, and we can construct the corresponding Pareto process by setting $Z=RY$, where $R$ is a  random variable with standard Pareto distribution and independent of $U$. For example, if $G$ is a continuous standard Gaussian process on $S$ and $U=\exp(G)$ is the associated log-Gaussian process, then the associated spectral process is obtained using $\tilde{Y}(s)=\exp(G(s)-\max_S G)$, $s\in S$. 
\end{example}

Next, we state elementary properties of the simple Pareto process,
the first one being the form of its finite-dimensional marginal distributions. For $s_1,\ldots, s_k\in S$ and $r_1,\ldots,r_k>0$,  we compute the multivariate survival function
\begin{align*}
\Pr(Z(s_1)>u_1,\ldots,Z(s_k)>u_k)&=\Pr(R>\max_{1\leq i\leq k} u_i/Y(s_i))\\
&=\E[\min(1,u_1^{-1}Y(s_1),\ldots,u_k^{-1}Y(s_k))].
\end{align*}
Similar formulas can be obtained for events of the form $\{Z> f\}$ or $\{Z\leq f\}$ for $f\in C^+(S)$:
\begin{align*}
\Pr(Z>f)&=\E\left[\min\left(1,\min_{s\in S}\frac{Y(s)}{f(s)}\right)\right],\\
\Pr(Z\leq f)&=1-\E\left[\min\left(1,\max_{s\in S}\frac{Y(s)}{f(s)}\right)\right].
\end{align*}

\index{POT-stability}

The second important property is the so-called POT-stability. For a measurable subset $A\subset C^+(S)$ and $u>1$, the simple Pareto process satisfies
\[
\Pr(u^{-1}Z\in A \mid \|Z\|>\theta u)=\Pr(Z\in A ).
\]
This states that conditionally on the exceedance $\|Z\|>u \theta$, the rescaled Pareto process $u^{-1}Z$ has the same distribution as the original process $Z$. To prove this, it suffices to verify the equality   for sets of the form 
\[
A=\{f\colon \theta^{-1}\|f\|> v, \theta f/\|f\|\in B\} \quad\mbox{with $v>0$ and $B\subset C_\theta^+(S)$.}
\]
For such sets, we can compute
\begin{align*}
\Pr(u^{-1}Z\in A \mid \|Z\|>u\theta)&=\Pr(R>uv,Y\in B \mid R>u)\\
&= \Pr(R>v)\Pr(Y\in B)\\
&=\Pr(Z\in A).
\end{align*}
This POT-stability is the reason why Pareto processes play a central role in the POT approach, in the same way as max-stability is a central property in the BM approach. To wit, it can be proven that a nondegenerate stochastic process satisfying the POT-stability property is (up to a power transform) a Pareto process in the sense of Definition~\ref{ch16:def1} -- see \cite[Theorem 2.1]{FdH14} or \cite[Theorem 2]{DR15} for precise statements.

The third important property is the following homogeneity property, which forms the basis for joint tail extrapolation. It states that for each measurable subset $A\subset \{f\in C^+(S)\colon \|f\|\geq \theta\}$, we have
\[
\Pr(Z\in uA)= u^{-1}\Pr(Z\in A),\quad u>1.
\]
The proof is straightforward from POT-stability, because the assumption on $A$ implies
\begin{align*}
\Pr(Z\in uA)&=\Pr(Z\in uA,\|Z\|>u\theta) \\
&=\Pr(\|Z\|>u\theta)\Pr(u^{-1}Z\in A\mid \|Z\|>u\theta)\\
&=u^{-1}\Pr(Z\in A).
\end{align*}
This property is useful for extrapolation in practice since for large $u$, the event $uA$ often defines an extreme event for which $\{Z\in uA\}$ has not been observed, so that its probability cannot be estimated empirically. In view of the above equality, its probability is the product of the small number $u^{-1}$ and the probability of the non-extreme event $\{Z\in A\}$, where the latter can be estimated directly through its relative frequency within a sample of observations.

We now introduce the fundamental theorem justifying the introduction of  the simple Pareto process and its use in spatial extreme value theory. In a nutshell, it states the equivalence of the approximation used in the BM and POT methods and the regular variation (RV) assumption. We consider here a standardized non-negative continuous random field $X=(X(s))_{s\in S}$, \emph{i.e.}, we assume standard Pareto margins, 
\begin{equation}\label{eq:standard-Pareto-margins}
\Pr(X(s)>u)=u^{-1} \quad\mbox{for all $u>1$ and $s\in S$}.
\end{equation}
The case of general margins will be considered in the next section.

\begin{theorem}\label{ch16:thm1}
Let $X=(X_s)_{s\in S}$ be a standardized continuous random field. Then the following statements are equivalent:
\begin{enumerate}
    \item[(RV)]  there exists a measure $\nu$ on $C^+(S)$ such that 
    \begin{equation}\label{eq:measure-nu}
    n\Pr(X\in nA)\longrightarrow \nu(A) \quad\mbox{as $n\to\infty$}
    \end{equation}
    for all measurable subsets $A\subset C^+(S)$ that are bounded away from $0$\footnote{The measurable set $A\subset C^+(S)$ is bounded away from $0$ if $\inf_{f\in A} \|f\|>0$.} and continuity sets of $\nu$\footnote{A set $A$ is a continuity set of $\nu$ if $\nu(\partial A)=0$ with $\partial A$ denoting the boundary of $A$.};
    \item[(BM)]  $X$ belongs to the maximum domain of attraction of a continuous simple max-stable random field $\eta$, i.e.  
    \[
    n^{-1}\max(X_1,\ldots,X_n)\stackrel{d}\longrightarrow \eta \quad \mbox{in $C^+(S)$ as $n\to\infty$},
    \]
    where $X_i$, $i\geq 1$, are independent copies of $X$ and the maximum is taken pointwise;
    \item[(POT)]  the normalized exceedance of $X$ over a high threshold converges to a simple Pareto process $Z$, \emph{i.e.}, 
   \begin{equation}\label{eq:conv-pp}
    u^{-1} X\,\big|\, \|X\|>u\theta \stackrel{d}\longrightarrow Z \quad\mbox{in $C^+(S)$ as $u\to\infty$}.
    \end{equation}
\end{enumerate}
The three limiting objects are closely related:  the exponent measure of $\eta$ in (BM) is equal to the measure $\nu$ appearing in (RV); the distribution of the Pareto process $Z$ in (POT) is characterized by
\[
\Pr(Z\in \cdot)=\frac{\nu(\cdot\cap \{f\colon \|f\|>\theta\})}{\nu(\{f\colon\|f\|>\theta\})}\quad\mbox{with}\quad \theta=\nu(\{f\colon\|f\|>1\}),
\]
and the spectral measure $\rho$ is given as 
\[
\rho(\cdot)=\frac{\nu(\{f\colon \theta f/\|f\|\in \,\cdot\,,\,\|f\|>\theta)\}}{\nu(\{f\colon\|f\|>\theta\})}. 
\]
\end{theorem}

A consequence of the preceding Theorem is that there is a one-to-one association of simple Pareto processes and simple max-stable processes. This association is made explicit with the spectral representation of simple max-stable processes \citep{deHaan1984}. Any continuous simple max-stable process $\eta$ on $S$ can be represented in the form
\begin{equation}\label{eq:spectral-construction}
\eta(s)=\max_{i\in\mathbb{N}} \Gamma_i^{-1} Y_i(s), \quad s\in S,
\end{equation}
where $(\Gamma_i)_{i\in\mathbb{N}}$ are the points of a homogeneous point process on $(0,\infty)$ and, independently, $(Y_i)_{i\in\mathbb{N}}$ are i.i.d. non-negative continuous random processes such that $\E[Y_i(s)]=1$ for all $s\in S$, and $\|Y_i\|=\theta$ is almost surely constant. The associated simple Pareto process is $Z(s)=RY(s)$, where $Y$ has the same distribution as the processes $Y_i$. 

\begin{example}\label{ex:BR} As an illustration, one can derive the form of the Pareto process associated with the celebrated Brown-Resnick class \citep{KSdH09}, or more generally with the class of max-stable processes associated with log-Gaussian processes. These max-stable processes can be represented in the form
\[
\eta(s)=\max_{i\geq 1} \Gamma_i^{-1} \tilde Y_i(s), \quad s\in S,
\]
where $(\Gamma_i)_{i\geq 1}$  is as above and $(\tilde Y_i)_{i\geq 1}$ are i.i.d. continuous log-Gaussian processes, such that $\E[\tilde Y_i(s)]=1$ for all $s\in S$. More precisely, one can write  $\tilde Y_i(s)=\exp(G_i(s)-\sigma^2(s)/2)$ where $G_i$ are continuous  Gaussian processes with zero mean and variance function $\E[G_i^2(s)]=\sigma^2(s)$. In this representation, the spectral process $\tilde Y$ is a log-Gaussian process that does not have a constant sup-norm. 

The associated Pareto process has a spectral function $Y$ with norm $\theta=\E[\|\tilde Y_i\|]$ and distribution characterized by
\[
\Pr(Y\in A)=\theta^{-1}\E\left[\|\tilde Y\|\times {1}_{\{ \theta \tilde Y/\|\tilde Y\|\in A\}}\right],\quad A\subset C_\theta^+(S) \mbox{ measurable}.
\]
The distribution of the spectral process $Y$ is fully determined by the probability distribution of the log-Gaussian process $\tilde Y$, but still it is quite involved to simulate according to this distribution \citep{Oesting2018};  see Section~\ref{sec:simulation} below.
\end{example}

\subsection{The generalized Pareto process}\label{sec:gpp}
\index{generalized Pareto process}

Following \citet{FdH14}, the notion of a \emph{generalized Pareto process} is deduced from the notion of simple Pareto process by a suitable transformation of the marginal distributions. Let $Z$ be a simple Pareto process with spectral measure $\rho$ on $C_\theta^+(S)$, and let $\mu,\sigma,\xi$ be continuous functions on $S$, with $\sigma$ positive. The generalized Pareto process $Z_{\mu,\sigma,\xi}$ is defined by
\[
Z_{\mu,\sigma,\xi}(s)=\mu(s)+\sigma(s)\frac{Z(s)^{\xi(s)}-1}{\xi(s)},\quad s\in S,
\]
with the usual convention $\frac{Z(s)^{\xi(s)}-1}{\xi(s)}=\log(Z(s))$ when $\xi(s)=0$. 

Its marginal distribution at location $s$ is given by
\[
\Pr(Z_{\mu,\sigma,\xi}(s)>u)=\Pr\Big[Z(s)>\Big(1+\xi(s)\frac{u-\mu(s)}{\sigma(s)}\Big)^{1/\xi(s)}\Big],
\]
of which the tail is equal to 
\[
\Big(1+\xi(s)\frac{u-\mu(s)}{\sigma(s)}\Big)^{-1/\xi(s)} \quad\mbox{for $u>\mu(s)+\sigma(s)\frac{\theta^{\xi(s)}-1}{\xi(s)}$}.
\]
Hence, the tail function is ultimately equal to the tail of the Generalized Pareto distribution with parameters $(\mu(s),\sigma(s),\xi(s))$.

The generalized Pareto process is justified by a general result stating the equivalence of BM and POT approaches.  Theorems 3.1 and 3.2 in \citet{FdH14}  generalize Theorem~\ref{ch16:thm1} to random processes with general margins that are necessarily standardized, and are summarized in the following single theorem.

\begin{theorem}
Let $X=(X(s))_{s\in S}$ be a continuous random field, and consider independent copies $X_i$, $i\geq 1$, of $X$. The following two statements are equivalent:
\begin{enumerate}
    \item[(BM)] There exist continuous functions  $a_n>0$ and $b_n$, $n\geq 1$, defined on $S$, such that
    \[
    \Big(\max_{1\leq i\leq n}\frac{X_i(s)-b_n(s)}{a_n(s)}\Big)_{s\in S}\stackrel{d}\longrightarrow (\eta(s))_{s\in S}\quad \mbox{as $n\to\infty$}.
    \]
    \item[(POT)] There exist continuous functions $\tilde{a}_u>0$ and $\tilde{b}_u$, $u>0$, defined on $S$, such that 
    \[
    \lim_{u\to\infty} \Pr\big(X(s)>\tilde b_u(s)\ \mbox{for some $s\in S$}\big)=0
    \]
    and
    \begin{equation}\label{eq:conv-gpp}
   \frac{X-\tilde b_u}{\tilde a_u} \ \Big|\ \{ X(s)>\tilde b_u(s) \ \mbox{for some $s\in S$}\} \stackrel{d}\longrightarrow Z, \quad \mbox{$u \to\infty$}.
    \end{equation}
\end{enumerate}
Then, the limit process $\eta$ in (BM) is a continuous max-stable process and the limit process $Z$ in (POT) is a generalized Pareto process. Furthermore, the two processes share the same index function $\gamma(s)$ which corresponds to the extreme value index of the marginal distribution of $\eta(s)$, $Z(s)$ or $X(s)$.
\end{theorem}

\subsection{Risk functionals and $r$-Pareto processes}

\index{risk functional}
\index{$r$-Pareto process}

As discussed in the introduction, the notion of exceedance in a multivariate or spatial setting is not unique, and several notions could be useful to the practitioner. A natural idea introduced in \citet{DR15,dFD18,dFD22} is to define the notion of an extreme event through a risk functional. A risk functional is defined on the space $C(S)$ of continuous functions on $S$ and is a function $r\colon C(S)\to \mathbb{R}$ that assigns to each spatial configuration $f=(f(s))_{s\in S}$ a risk $r(f)$.  A realization $f$ is considered as extreme when its risk $r(f)$ is large. 

For example, \citet{Buishand2008} are interested in the total amount of rainfall over a catchment $S$, \emph{i.e.}, $r(f)=\int_{S} f(s)\mathrm{d}s$, where $f(s)$ represents the amount of precipitation at location $s\in S$. More generally, for spatial extreme rainfall, \citet{dFD18} consider risk functionals of the form $r_p(f)=\left(\int_{S} |f(s)|^p\mathrm{d}s\right)^{1/p}$ with $p>0$; for large $p$, the risk functional get closer to the maximum risk functional $\max_{s\in S}|f(s)|$, while smaller values of $p$ allow all locations to contribute to the value of the risk functional in a more balanced way. To model extreme heatwaves over Europe, \citet{K2024} consider a risk functional $r$  as the spatial mean of temperature anomalies $r(f)=\lvert S\rvert^{-1}\int_{S} f(s)\mathrm{d}s$. Another example concerns the situation where a  specific location $s_0\in S$ is of interest with the risk functional chosen as $r(f)=f(s_0)$.

The theory of $r$-Pareto processes \citep{DR15} and Generalized $r$-Pareto processes \citep{dFD18,dFD22} considers convergence of the normalized exceedances
\[
\frac{X-b_u}{a_u}\,\Big|\, r(X)>a_u \stackrel{d}\longrightarrow Z, \quad \mbox{$u \to\infty$}.
\]
The theory provides general conditions on the stochastic process $X$ for such a convergence to hold and characterizations of the limit process $Z$. 

Relatively simple conditions can be given in the framework of regularly varying stochastic processes and homogeneous risk functionals. A continuous stochastic process is called regularly varying on $C(S)$ with exponent $\alpha>0$ and spectral measure $\sigma$ (a probability measure on the unit sphere $C_1(S)$), noted $X\in \mathrm{RV}_\alpha(C(S),\sigma)$, if 
there exists a normalisation sequence $(b_n)$ such that
\[
n\Pr\left(X/\|X\|\in A,\,\|X\|>ub_n\right)\longrightarrow u^{-\alpha}\sigma(A),\quad n \rightarrow\infty,
\]
for all $u>0$ and  $A\subset C_1(S) $ (continuity set of $\sigma$). Necessarily, the sequence $b_n$ satisfies $n\Pr(\|X\|>b_n)\to 1$, and we also have the convergence
\[
n\Pr\left(X/b_n\in\cdot\right)\longrightarrow \nu(\cdot), \quad n\rightarrow \infty,
\]
where the limit measure $\nu$ is homogeneous of order $-\alpha$, \emph{i.e.}, $\nu(tA)=t^{-\alpha}\nu(A)$ for any $t>0$ and measurable set $A\subset C(S)$, and has infinite mass (with explosion of the mass around the function $0$).

\begin{theorem}[\citealt{DR15}]\label{thm:Dombry-Ribatet}
Assume that $X\in \mathrm{RV}_\alpha(C(S),\sigma)$. Let  $r:C(S)\to [0,\infty)$ be a risk functional assumed to be $1$-homogeneous, \emph{i.e.}, $r(uf)=ur(f)$ for all $u>0$. If $r$ is continuous at $0$ and does not vanish $\sigma$-almost-everywhere, then the normalized $r$-exceedances of $X$ converge in distribution:
 \begin{equation}\label{eq:conv-rpp}
u^{-1}X\,\big|\, r(X)>u\stackrel{d}\longrightarrow Z, \quad \mbox{as $u\to\infty$},
\end{equation}
where $Z\stackrel{d}=RY$ with $R$ and $Y$ independent, $R$ possessing a standard Pareto distribution with index $\alpha$ (\emph{i.e.}, $\Pr(R>u)=u^{-1/\alpha}$, $u>1$) and  $Y$  a distribution characterized by
\begin{equation}\label{eq:r-pp-spectral}
\Pr(Y\in \cdot\,)= \frac{\int r(f)^\alpha 1_{\{f/r(f)\in \,\cdot\,\}}\sigma(\mathrm{d}f)}{\int r(f)^\alpha \sigma(\mathrm{d}f)}.
\end{equation}
Equivalently, we have
\begin{equation*}
\Pr(Z\in \cdot\,)= \frac{\nu(\,\cdot\cap\{f\colon r(f)>1\})}{\nu(\{f\colon r(f)>1\})}.
\end{equation*}
\end{theorem}

The form of the limit distribution implies that  $r(Y)=1$ almost surely, so that $r(Z)=R$ has a standard Pareto distribution with index $\alpha$. The condition that $r$ is non-negative is not restrictive, because one can always replace $r(f)$ by $r_+(f)=\max(r(f),0)$. Also, the result can be adapted for $\beta$-homogeneous risk functionals with $\beta>0$, because one can always replace $r(f)$ by $r(f)^{1/\beta}$, where $r^{1/\beta}$ is $1$-homogeneous.

\begin{example}\label{ex:risk-at-one-location}
Specific models for regularly varying random fields with index $\alpha$ are max-stable random processes with $\alpha$-Fr\'echet margins, and one can consider the corresponding $r$-Pareto processes. Assume $\eta$ is an $\alpha$-Fréchet max-stable process represented in the form
\[
\eta(s)= \max_{i\in\mathbb{N}}\Gamma^{-1/\alpha}\tilde Y_i(s),\quad s\in S,
\]
with $(\Gamma_i)_{i\in\mathbb{N}}$ the points of a homogeneous point process on $(0,\infty)$, and $(\tilde Y_i)_{i\geq 1}$ i.i.d. copies of a non-negative random field $\tilde Y$ satisfying $\E[\tilde Y(s)^\alpha]<\infty$. The $r$-Pareto process takes the form $Z=RY$ as in  Theorem~\ref{thm:Dombry-Ribatet}, where Equation~\eqref{eq:r-pp-spectral} can be rewritten as
\[
\Pr(Y\in A)= \frac{\E\left[r(\tilde Y)^\alpha  1_{\{\tilde Y/r(\tilde Y)\in A\}}\right]}{\E\left[r(\tilde Y)^\alpha\right]}.
\]
For the simple max-stable process considered in Example~\ref{ex:BR}, an important case arises for the risk functional $r(f)=f(s_0)$. Here, $\alpha=1$ and $\tilde Y(s)=\exp(G(s)-\sigma^2(s)/2)$ and, according to \citet[Proposition 6]{DEO16},  $Y$ is again a log-Gaussian process  given by 
\begin{equation}\label{eq:Y-s_0}
Y(s)=\exp\left(G(s)-G(s_0)-\frac{1}{2}\mathrm{Var}\Big(G(s)-G(s_0)\Big)\right).
\end{equation}
Similarly, the class of extremal-$t$ processes \citep{O2013} also gives rise to explicit formulas, see \citet[Proposition 7]{DEO16} and \citet{TO15}.
\end{example}

\subsection{Generalized $r$-Pareto processes}
In this section, we finally discuss the approach of  \citet{dFD18} that extends the construction of $r$-Pareto processes and convergence to such processes to a general extreme value index $\xi\in\mathbb{R}$ (positive, null or negative). In their work, the extreme value index $\xi$ is assumed constant over the domain $S$ (unlike in Section~\ref{sec:gpp}), the risk functional is not necessarily homogeneous, and a generalization of the condition of functional regular variation is used. The random field $X$ is assumed to satisfy the generalized regular-variation condition $X\in\mathrm{GRV}(\xi,a_n,b_n,\nu)$, which means that  the following two properties hold: the $1$-dimensional margins satisfy, for all $s\in S$,
\[
\lim_{n\to\infty} n\Pr\left(\frac{X(s)-b_n(s)}{a_n(s)}>x \right)=(1+\xi x)_+^{-1/\xi},
\]
 where $x_+=\max(x,0)$ denotes the positive part of a scalar $x$. If furthermore the appropriately normalized process satisfies the functional convergence 
\[
\lim_{n \to\infty} n\Pr\left(\Big(1+\xi\frac{X-b_n}{a_n}\Big)_+^{1/\xi} \in \ \cdot\  \right)=\nu(\cdot)\quad \mbox{in $C^+(S)\setminus\{0\}$}.
\]
We use the  usual convention $(1+\xi z)^{1/\xi}=\exp(z)$ when $\xi=0$.

These assumptions allow for the existence of a limit for the process $(X-b_n)/a_n$ conditioned on exceedances of the risk $r\big((X-b_n)/a_n\big)\geq u$.
Here the risk is evaluated on the normalized process, which is somewhat unnatural since it is not the scale of observations, and the normalizing functions are \emph{a priori} unknown. In order to push the analysis further, convenient and mild assumptions are stated in terms of the asymptotic proportionality of the normalizing function $a_n(s)\sim   \bar a_n A(s)$ and of the associated risk, $r(a_n)\sim  \bar a_n$ as $n\to\infty$. Under such considerations,  \citet[Theorem 1]{dFD22} obtain 
\[
 \left\lfloor \frac{X-b_n}{a_n}\right\rfloor \ \Big| \ r\Big(\frac{X-b_n}{r(a_n)}\Big)> u  \stackrel{d}\longrightarrow Z, \quad n\rightarrow \infty,
\]
where $Z$ is called the \emph{generalized $r$-Pareto process} and $\lfloor\cdot\rfloor$ denotes a truncation operation handling possible small values. We refer to  the original paper for a precise statement; here, we rather focus on the structure of the limit, which is important for modeling purposes.

\begin{definition}[\citealt{dFD22}]
Let $\xi\in\mathbb{R}$, $a>0$ and $b$ be continuous functions on $S$ and  $\nu$ be a $-1$-homogeneous measure on $C^+(S)\setminus\{0\}$. Assume that the risk functional $r:C(S)\to\mathbb{R}$ is valid in the sense that
\[
\mathcal{A}=\left\{f\in C^+(S)\colon r\left( A\frac{f^\xi-1}{\xi}\right)\geq 0\right\}
\]
satisfies $\nu(\mathcal{A})\in (0,\infty)$, where $A=a/r(a)$.
The associated generalized $r$-Pareto process is 
\[
Z(s)=a(s)\frac{Y(s)^\xi -1}{\xi}+b(s),\quad s\in S,
\]
where $Y$ is the stochastic process on $\mathcal{A}$ 
with distribution $\nu(\cdot \cap \mathcal{A})/\nu(\mathcal{A})$.
\end{definition}

If we assume that for some $s_0\in S$, $u_0>0$, the following implication,
\[
r((f-b)/r(a))> 0 \ \Rightarrow r((tf-b)/r(a))> 0, 
\]
holds for all $t>1$ and  $f$ such that $f(s_0)\geq u_0$, then the marginal distribution of $Z$ at $s_0$ satisfies 
\[
\Pr\left(Z(s_0)>u\mid Z(s_0)>u_0 \right)=\left(1+\xi \frac{u-u_0}{\sigma(s_0)} \right)^{-1/\xi},\quad u>u_0,
\]
with scale parameter $\sigma(s_0)=a(s_0)+\xi(u-b(s_0))$, \text{i.e.}, the tail of $Z(s_0)$ is of generalized Pareto type.

\section{Models, Simulation and Statistical Inference}\label{sec:Models-etc}

Generalized $r$-Pareto processes and their various subclasses described above can be used as statistical models for spatial extreme-event episodes and can serve different purposes. It is possible to estimate model parameters to gain better knowledge on the behavior of joint spatial extremes, for example spatial range parameters characterizing spatial extent and coefficients $\theta$ characterizing occurrence frequency of risk exceedances in terms of the functional $r$. Another goal can be to estimate probabilities of various joint extreme events. Furthermore, models can be used for stochastic simulation of  new extreme-event episodes, potentially much more extreme than those available in the data. For example, stochastic extreme-weather generators can produce weather input data for models that are used to assess extreme-event impacts in risk assessment studies.

\subsection{Pareto process models}

\index{Brown--Resnick model}

Many commonly used models for spatial extremes are most naturally formulated in terms of the constructive representation of max-stable processes through the spectral construction \eqref{eq:spectral-construction}, from which we can derive representations for $r$-Pareto processes for practically relevant choices of the risk functional $r$. For processes of Brown--Resnick type \citep{KSdH09} already discussed in Examples~\ref{ex:BR} and \ref{ex:risk-at-one-location}, the dependence structure is fully characterized by the variogram function of the Gaussian process $G_1$, given by $\gamma(s_1,s_2)=\E[((G_1(s_1)-G_1(s_2))^2]$ for  any two locations $s_1,s_2\in S$; for stationary and isotropic models of the variogram function, it depends only on the Euclidean distance $h=\|s_1-s_2\|$. The Pareto processes themselves are always nonstationary since they are defined on a fixed domain $S$ using a risk functional whose values also depend on $S$, and in some cases even on specific fixed locations, \emph{e.g.}, on $s_0\in S$ in $r(f)=f(s_0)$.    

Parametric models can be obtained by choosing a parametric variogram function, such as the classical choice of using a power variogram $\gamma(h) = (h/\beta)^\alpha$ with scale $\beta>0$ and shape $\alpha\in(0,2)$. The power variogram is unbounded, and as a consequence the bivariate extremal coefficients, given by the normalizing constant $\theta(h)=\Phi(\sqrt{\gamma(h)/2})$ if $\mathcal{S}=\{s_1,s_2\}$ and $r=\max$, can take any value in $[1,2)$, with $1$ indicating perfect asymptotic dependence and $2$ asymptotic independence. The Gaussian process $G_1$ in \eqref{eq:spectral-construction} associated with the power variogram is the two-dimensional fractional Brownian motion, \emph{i.e.}, a Gaussian process with stationary increments. For  variogram functions with a finite upper bound, such as $\gamma(h)=2(\rho(0)-\rho(h))$ corresponding to a stationary Gaussian process $G_1$ with mean zero and covariance function $\rho$, the extremal coefficient $\theta(h)$ remains bounded away from $2$ as $h\rightarrow\infty$. 

Another flexible class of models are elliptical $r$-Pareto processes \citep{TO15}, resulting from setting $Y_i = (G_i)_+^\alpha/\E[(G_i)_+^\alpha]$ in the spectral construction \eqref{eq:spectral-construction} of the corresponding max-stable processes known as  \emph{extremal-$t$ processes} \citep{O2013}. In these processes, a positive probability mass  $\Pr(Z(s)=0)>0$ arises in the simple Pareto process $Z$ (and therefore also at the lower marginal endpoints of the corresponding generalized Pareto processes) and must be taken into account in inference procedures \citep{TO15}.  

Finally, the construction principle described in Example~\ref{ex:construction-with-normalized-u} is very general and flexible since a large variety of processes $U$ could be used in the construction, but it has so far not been widely applied in practice. Often, already existing max-stable models were used as the starting point for obtaining $r$-Pareto models. 


\subsection{Simulation of Pareto processes}\label{sec:simulation}

This section provides some algorithms, R code and figures for the simulation of Pareto processes. First, we present general methodology with a few theoretical results and  R code (using the \texttt{geoR} package \citep{geoR} for the simulation of spatial Gaussian fields), and we then introduce the recent \verb!mev! package \citep{B2023} that provides built-in functions for some of the most popular models. More simulations can be found in the R Markdown script in the supplemental material to this chapter.

The methods by which Pareto processes and the variants discussed above can be simulated depend on the choice of the risk functional $r$ and the structure of the measure $\nu$. 
Recall that the probability distribution of the $r$-Pareto process  for a given measure $\nu$ as introduced in \eqref{eq:measure-nu} is 
\[
 \frac{\nu(\, \cdot\, \cap A_1)}{\nu(A_1)} \quad\mbox{with}\quad A_{1} = \{f \in C^+(S) : r(f) \geq 1\},
\]
\emph{i.e.}, it is the measure $\nu$ truncated to keep only $r$-exceedances above threshold $1$ and renormalized into a probability measure. 

For some specific combinations of $\nu$ and $r$, the distribution of the Pareto process corresponds to a well-known distribution that is easy to simulate from. For example, in the Brown--Resnick and extremal-$t$ models with risk function $r(f) = f(s_0)$ for a fixed location $s_0\in S$, the spectral function $Y$ in the $r$-Pareto process $Z=RY$ is a log-Gaussian or (marginally tranformed) Student's $t$ process, respectively; see Example~\ref{ex:risk-at-one-location} above and \citet{EMKS15,TO15}. Then, simulation can be performed using standard tools for Gaussian process simulation, for example with the \texttt{geoR} package as in the following R code.

\begin{verbatim}
grid = as.matrix(expand.grid(1:20, 1:20))
s0 = 110 # Index of conditioning location
distmat = as.matrix(dist(grid)) # Matrix of distances
# Define the power variogram
vario = function(h,beta=1,alpha=1.5){(h/beta)^alpha}
gamma = vario(distmat) # Variogram matrix
# Simulation of a Gaussian random field:
set.seed(123)
G = geoR::grf(grid = grid, cov.pars = c(1, 1.5),
              cov.model="power")$data
# Compute the corresponding spectral process:
Y = exp(G-G[s0]-gamma[,s0])
# Simulation of a standard Pareto variable:
R = evd::rgpd(n = 1, loc = 1, scale = 1, shape = 1)
# Construct the simple Pareto process:
Z = R*Y 
\end{verbatim}

\begin{figure}[htb]
\begin{center}
\includegraphics[scale=0.35]{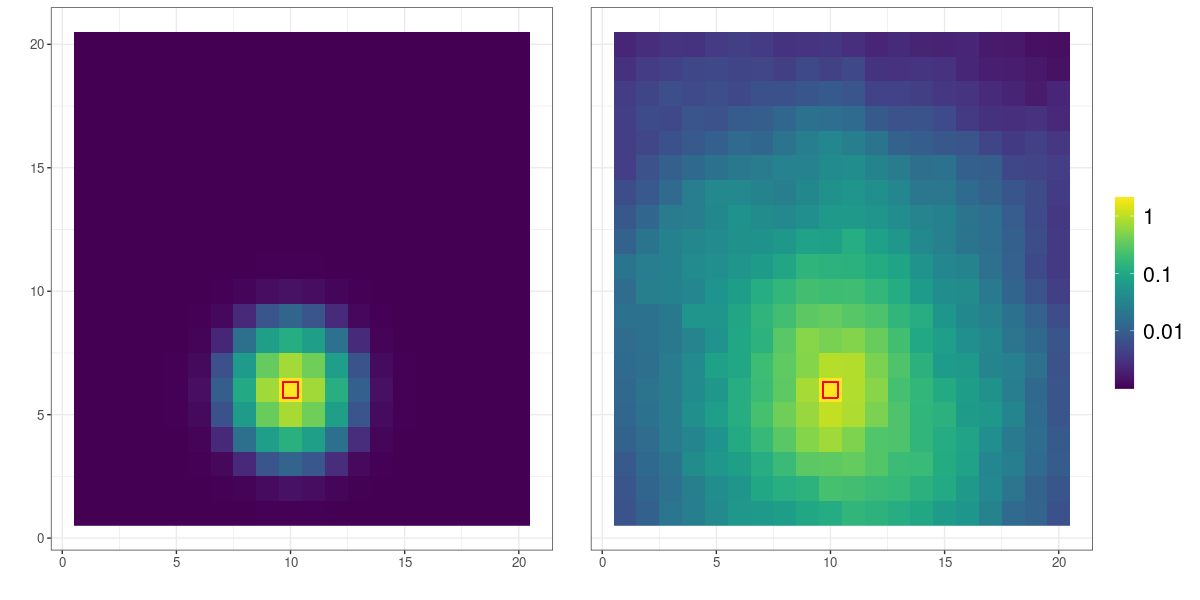}
\end{center}
\caption[BR site]{One realization on logarithmic scale of an $r$-Pareto process (\emph{i.e.}, with $\mu(s)=0$, $\sigma(s)=1$ and $\gamma(s)=0$ in the generalized $r$-Pareto process)  with $r(f)=f(s_0)$, using a Brown-Resnick model and a power variogram defined as $\gamma(h)=(h/\beta)^\alpha$. Left panel: $\beta=1$, $\alpha=1.5$. Right panel: $\beta=1.5$, $\alpha=0.8$. The red square highlights the conditioning site $s_0$.}
\label{fig::BR_site}
\end{figure}

Elaborating on this simple example, one can derive the distribution for simulation of the $r$-Pareto process associated with risk functionals of the more general form
\begin{equation}\label{eq:risk-l1}
r(f)=\sum_{k=1}^K \pi_k f(s_k)\quad \mbox{or}\quad r(f)=\int_S  f(s) \,\pi(s)\mathrm{d} s,
\end{equation}
where $\pi\geq 0$ is a probability distribution, \emph{i.e.}, $\sum_{k=1}^K \pi_k = 1$ or $\int_S \pi(s) \mathrm{d}s = 1$, respectively. 
The following result is closely related to simulation algorithms for Brown-Resnick processes based on the spectral decomposition with respect to the $L^1$-norm, see \citet{DM15} and \citet{DEO16}. 

\begin{proposition}\label{prop:simu-l1}
In the setting of Theorem~\ref{thm:Dombry-Ribatet}, assume that $X\in\mathrm{RV}_\alpha(C^+(S),\sigma)$ with $\alpha=1$. Denote by $Z^{(s)}=RY^{(s)}$ the Pareto process associated with the risk functional $r_s(f)=f(s)$ and, similarly, let   $Z^{(r)}=RY^{(r)}$ be the Pareto process associated with the risk functional $r$ defined by the convex combination in  \eqref{eq:risk-l1}. Then, 
\[
Y^{(r)}\stackrel{d}=\frac{Y^{(\tau)}}{r(Y^{(\tau)})}
\]
with $\tau\sim \tilde\pi$ representing a random sampling point with distribution
\[
(\tilde\pi_k)_{1\leq k\leq K}=\left(\frac{\pi_k\int f(s_k)\sigma(\mathrm{d}f)}{\int r(f)\sigma(\mathrm{d}f)}\right)_{1\leq k\leq K}
\]
or
\[
(\tilde\pi(s))_{s\in S}=\left( \frac{\pi(s)\int f(s)\sigma(\mathrm{d}f)}{\int r(f)\sigma(\mathrm{d}f)}\right)_{s\in S},
\]
in the discrete and continuous case, respectively.
\end{proposition}
\begin{proof} We outline the proof of Proposition~\ref{prop:simu-l1} for the discrete case. According to Equation~\eqref{eq:r-pp-spectral}, the spectral measure associated with the risk functional $r_s$ is
\[
\Pr(Y^{(s)}\in \cdot\,)= \frac{\int f(s) 1_{\{f/f(s)\in \,\cdot\,\}}\sigma(\mathrm{d}f)}{\int f(s) \sigma(\mathrm{d}f)},
\]
and the spectral measure associated with the risk functional $r$ is
 \[
\Pr(Y^{(r)}\in \cdot)= \frac{\int r(f) 1_{\{f/r(f)\in \,\cdot\,\}}\sigma(\mathrm{d}f)}{\int r(f) \sigma(\mathrm{d}f)}.
\]
In the discrete case, we have $r(f)=\sum_{k=1}^K \pi_k f(s_k)$, so that
\begin{align*}
\Pr(Y^{(r)}\in \,\cdot\,)&= \frac{\sum_{k=1}^K \pi_k\int f(s) 1_{\{f/r(f)\in \,\cdot\,\}}\sigma(\mathrm{d}f)}{\int r(f) \sigma(\mathrm{d}f)}  \\
&= \sum_{k=1}^K \frac{\pi_k\int f(s_k)}{\int r(f) \sigma(\mathrm{d}f)} \Pr\left(Y^{(s_k)}/r(Y^{(s_k)})\in \,\cdot\,\right)\\
&= \sum_{k=1}^K \tilde\pi_k \Pr\left(Y^{(s_k)}/r(Y^{(s_k)})\in \,\cdot\,\right).
\end{align*}
Therefore, the distribution of $Y^{(r)}$ is a mixture of the distributions of $Y^{(s_k)}/r(Y^{(s_k)})$, $k=1,\ldots, K$, with mixture distribution $\tilde\pi_k$, $k=1,\ldots, K$, which provides the representation with a random sampling point $\tau$. 
\end{proof}

Proposition~\ref{prop:simu-l1} allows simulation in the case of linear risk functionals, provided that simulation according to $Y^{(s)}$ is feasible. The following code continues the preceding example with the simulation of the Pareto process associated with the risk functional $r(f)=K^{-1}\sum_{k=1}^K f(s_k)$, where $S=\{s_1,\ldots,s_K\}$ denotes the simulation grid. Since we use a stationary variogram function, $\tilde{\pi}_k=1/K$ for all $k=1,\ldots,K$, \emph{i.e.}, the sampling weight is the same for all locations.    

\begin{verbatim}
# Randomly choose conditioning location uniformly:
tau = sample(1:nrow(grid),1)
# Simulation of a Gaussian random field:
G = geoR::grf(grid = grid, cov.pars = c(1, 1.5),
              cov.model="power")$data
# Compute the corresponding spectral process:
Y = exp(G-G[tau]-gamma[,tau])
# Simulation of a standard Pareto variable:
R = evd::rgpd(n = 1, loc = 1, scale = 1, shape = 1)
# Construct the simple Pareto process:
Z = R*Y/sum(Y)
\end{verbatim}

One can further enlarge the class of risk functionals for which simulation is feasible by  rejection sampling in cases where it is relatively easy to sample according to a specific risk functional but one wants to generate samples according to another risk functional. This is the purpose of the next result and requires a domination condition of the risk.

\begin{proposition}\label{prop:simu-rejection}
Let $r_1$ and $r_2$ be two $1$-homogeneous risk functionals such that $r_2\leq M r_1$ with a positive and finite constant $M$.  If simulation of the Pareto process $Z^{(r_1)}$ is feasible, then $Z^{(r_2)}$ can be simulated by drawing independent copies of $Z^{(r_1)}$ until $r_2( Z^{(r_1)})\geq M$ and then setting $Z^{(r_2)}=M^{-1}Z^{(r_1)}$.
\end{proposition}

\begin{proof} We assume it is possible to simulate  $Z^{(r_1)}$ with distribution
\[
\Pr(Z^{(r_1)}\in \,\cdot\,)=\frac{\nu(\,\cdot\,\cap A_1)}{\nu(A_1)}\quad \mbox{with}\quad A_1=\{f\colon r_1(f)\geq 1\},
\]
and the goal is to simulate $Z^{(r_2)}$ with distribution
\[
\Pr(Z^{(r_2)}\in \,\cdot\,)=\frac{\nu(\,\cdot\,\cap A_2)}{\nu(A_2)}\quad \mbox{with}\quad A_2=\{f\colon r_2(f)\geq 1\}.
\]
The homogeneity of the risk functionals, together with the bound $r_2\leq Mr_1$, imply the inclusion $A_2\subset M^{-1} A_1$, because
\[
f\in A_2\ \Rightarrow\ r_2(f)\geq 1 \ \Rightarrow\  Mr_1(f)\geq 1 \ \Rightarrow\ r_1(Mf)\geq 1  \ \Rightarrow\  Mf\in A_1.
\]
Then, the $(-\alpha)$-homogeneity of $\nu$ implies 
\[
\frac{\nu(\,\cdot\,\cap M^{-1}A_1)}{\nu(M^{-1}A_1)}=\frac{M^\alpha\nu((M\,\cdot\,)\cap A_1) }{M^\alpha\nu(A_1)}=\Pr(M^{-1}Z^{(r_1)}\in\,\cdot\, ).
\]
This shows that $M^{-1}Z^{(r_1)}$ has the same distribution as ``$\nu$ conditioned on $M^{-1}A_1$'' (with a slight abuse of language because $\nu$ is an infinite measure); conditioning further on $A_2\subset M^{-1}A_1$ yields the distribution of $Z^{(r_2)}$ and justifies the use of rejection sampling.
\end{proof}

It is often useful to  apply Proposition~\ref{prop:simu-rejection} by choosing $r_1=\frac{1}{|S|}\sum_{s\in S} f(s)$ with  $S$ a finite simulation grid with cardinality $|S|$. If it is possible to simulate according to $r(f)=f(s)$, then, according to Proposition~\ref{prop:simu-l1}, simulation of $Z^{(r_1)}$ becomes feasible. Standard $1$-homogeneous risk functionals satisfying $r_2\leq Mr_1$ include the following examples:
\begin{itemize}
\item[-] minimum: $r_2(f)=\min_{s\in S} f(s)\leq r_1(f)$, \emph{i.e.}, $M=1$; 
\item[-] maximum: $r_2(f)=\max_{s\in S} f(s)\leq |S|\, r_1(f)$, \emph{i.e.}, $M=|S|$;
\item[-] any order statistics $r_2(f)$ of the sample $\{f(s)\}_{s\in S}$, including the median, where we can set $M=|S|$ by analogy with the maximum;
\item[-] $l^p$-norm: $r_2(f)=(\sum_{s\in S} f(s)^p)^{1/p}\leq Mr_1(f)$ with $M=1$ if $p\in (0,1)$ and $M=|S|^{p-1}$ if $p>1
$.
\end{itemize}

The following code continues the previous examples with the simulation of the Pareto process associated with the maximum risk functional $r(f)=\max_{s\in S} f(s)$, using the rejection method.

\begin{verbatim}
Z1 = rep(0,nrow(grid))
while(max(Z1)<nrow(grid)){
  s0 = sample(1:nrow(grid),1)
  G = geoR::grf(grid = grid, cov.pars = c(1, 1.5),
                cov.model="power")$data)
  Y = exp(G-G[s0]-gamma[,s0])
  R = evd::rgpd(n = 1, loc = 1, scale = 1, shape = 1)
  Z1 = R*Y/mean(Y)
}
Z2 = Z1/nrow(grid)
\end{verbatim}

The R package \texttt{mev} \citep{B2023} allows for simulation of $r$-Pareto processes and generalized $r$-Pareto processes, with the functions \texttt{rparp} and \texttt{rgparp} respectively. The user can specify the risk functional with the argument \texttt{riskf}, where currently available risks includes maxima, minima, geometric average, location-conditioning value and $L^2$ norm, and a specific dependence model using the argument \texttt{model} (\emph{e.g.}, ``br" for the Brown-Resnick model) along with a user-defined semivariogram (argument \texttt{vario}) or a parameter matrix (argument \texttt{sigma}).

Figure \ref{fig::BR} depicts a simulated example of a Brown-Resnick model with power variogram and risk functional $r=\max$, using the following code. Visually, a clear difference with Figure~\ref{fig::BR_site} is that the spatial maximum is not located near a fixed conditioning location but can arise at any location of the simulation grid with the same probability. 

\begin{verbatim}
grid = as.matrix(expand.grid(1:20, 1:20))
nsim = 5 # Number of simulated observations
sim.max = mev::rparp(n=nsim, riskf = 'max', vario=vario,
    coord=grid, model='br')
\end{verbatim}

\begin{figure}[htb]
\begin{center}
\includegraphics[scale=0.35]{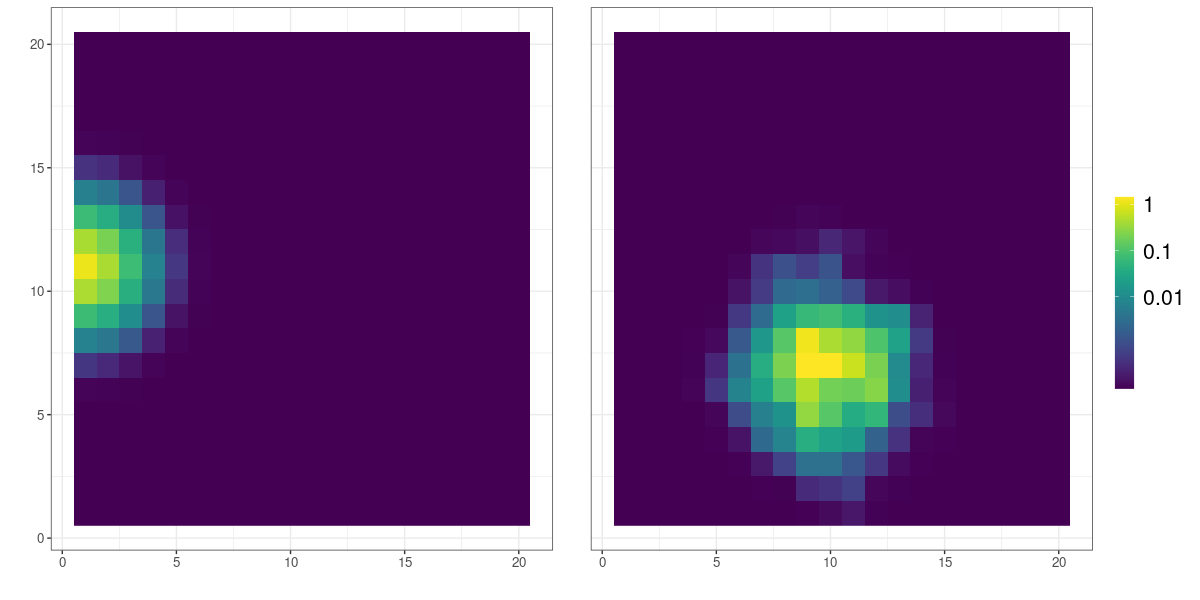}
\end{center}
\caption[BR mev]{One realization on logarithmic scale (see caption of Figure~\ref{fig::BR_site}) of an $r$-Pareto process with $r(X)=\max_{s\in S} f(s)$, using a Brown-Resnick model and a power variogram defined as $\gamma(h)= (h/\beta)^\alpha$ with $\beta=1$, $\alpha=1.5$. Left panel: using the procedure described in Proposition \ref{prop:simu-rejection}. Right panel: using the \texttt{R} package \texttt{mev}, also based on rejection sampling.}
\label{fig::BR}
\end{figure}

\bigskip


Finally, we outline the idea of directly resampling the extreme-risk events in a dataset at hand. It is possible to simulate a new sample of $r$-Pareto processes $R_kY^{(r)}_{k}$, $k=1,\ldots,K$, with potentially very large sample size $K$ and much higher risk than observed in the data, through resampling procedures that do not require any parametric assumptions and statistical inferences on the dependence structure. One proceeds by drawing new variables $R_k\sim \mathrm{Pareto}(1)$, whereas  empirical spectral processes $Y^{(r)}_{j}$, $j=1,\ldots,m$, are extracted from the data. The principle of this approach was described by \citet{FdH14}. \citet{PTCO20} implement it for ``lifting" extreme precipitation episodes to higher risk levels on a spatiotemporal domain $S$, and \citet{OAM21} further apply nonparametric spatial resampling techniques to obtain spectral processes $Y^{r}_{k}$ different from the observed ones but possessing similar spatial patterns.


\subsection{Statistical inference for Pareto processes}
\label{sec:stat-inference-pareto-processes}


We consider a setting with i.i.d. processes $X_i$ and data  $X_i(s_j)$  available  for times $i=1,\ldots,n$ and locations $s_j$, $j=1,\ldots, D$, and write $\mathbf X_i = (X_i(s_1),\ldots,X_i(s_D))^T$ for the data vector at time $i$. The following methods can be adapted to settings with time-varying data availability at locations, for example in partially missing data,  but to keep notation concise we will not address such extensions. 

We set $S_{\text{data}}=\{s_1,\ldots,s_D\}$ for statistical inference, such that the risk functional $r_{\text{data}}$ can be assessed for data, where $r_{\text{data}}(\mathbf{z})=\max_{j=1}^D z_j$ for the classical (generalized) Pareto process. Due to the  equivalence of the existence of  limits for various risk functionals $r$ discussed in \S~\ref{sec:theory-pareto-processes}, it is feasible to first estimate a model using a  specific  combination of $S_{\text{data}}$ and $r_{\text{data}}$ and then conduct simulation and prediction with different $S$ and $r$. For example, $S_{\text{data}}$ can contain irregularly spaced observation stations, whereas $\mathcal{S}$ can be a fine regular spatial pixel grid covering the study area. Nevertheless, $S_{\text{data}}$, $r_{\text{data}}$,  $S$ and $r$ must be chosen carefully when we  use asymptotic representations with finite-sample data, especially when estimated models  can suffer from certain biases that could be further amplified after switching to different $S$ and $r$. 

Generalized $r$-Pareto processes and their subclasses can be constructed and estimated by assuming equality of the left- and right-hand side in one of the limits  \eqref{eq:conv-pp}, \eqref{eq:conv-gpp} and \eqref{eq:conv-rpp}, for fixed $r$ or $u$, respectively. To fix an appropriate threshold $u$ for $r_{\text{data}}(\mathbf X_i)$, several aspects should be taken into account: the exceedance region $\{\mathbf z: r(\mathbf z)>u\}$ should include the extreme events for which we seek statistical predictions; the number of replicates $\bm X_i$ satisfying $r_{\text{data}}(\bm X_i)>u$ should be large enough to keep estimation uncertainties moderate; the POT-stability should be checked -- there should be no stochastic dependence between $R_i=r_{\text{data}}(\bm X_i)$ and $\bm X_i/r_{\text{data}}(\bm X_i)$, and known distributional properties of $R_i$ (\emph{e.g.}, being standard Pareto in the case of $r$-Pareto processes), should be verified. In many real-data applications, POT-stability is not well achieved in available data at observed levels of extremenes.
Nevertheless, the class of  asymptotic POT-stable models, \emph{i.e.} generalized $r$-Pareto processes, can provide useful predictions and a good working dependence model in practice, although estimations of joint-tail probabilities with strong extrapolation beyond the observed range of data may be biased.  

\subsubsection{Handling marginal distributions}

Subsequently, we focus on inference for $r$-Pareto processes where data are available on the standardized marginal scale, \emph{i.e.}, with marginal standard Pareto distributions as defined in \eqref{eq:standard-Pareto-margins}. The assumption of temporal stationarity of processes $X_i$ may require further data preprocessing or subsetting. Marginal distributions of data are allowed to be nonstationary and are usually not known beforehand, and a first modeling step is necessary to appropriately specify them and transform data to the standardized marginal scale.  In the case where generalized $r$-Pareto processes are used, the parameter fields $\mu(s)$, $\sigma(s)$ and $\xi(s)$ must be estimated \citep{PTCO20,dFD22}.  

A standard approach for estimating marginal distributions, motivated by univariate asymptotic theory, consists of using a Generalized Pareto (GP)  distribution  for excesses $(X_i(s_j)-u_j) \mid (X_i(s_j) >u_j)$ above an appropriately chosen high marginal quantile $u_j$, $j=1,\ldots,D$. With risk functionals $r_{\text{data}}$ for which the conditioning event in the relevant limit among \eqref{eq:conv-pp}, \eqref{eq:conv-gpp},  \eqref{eq:conv-rpp} includes values below one of the marginal thresholds, we also need to specify the marginal distributions below the thresholds $u_j$. The empirical distribution function of $X_i(s_j)$, $i=1,\ldots,n$, can be used below the marginal threshold  \citep{Coles1991}. 
When focus is primarily on estimating the dependence model of the $r$-Pareto process, often the empirical distribution functions are used for both body and tail of marginal distributions, without specifying a GP distribution for the tail.

\subsubsection{Maximum likelihood}

With a slight abuse of notation, we write $X_i(s_j)$ for the marginally standardized data, \emph{i.e.}, possessing standard Pareto distribution $\Pr(X_i(s_j)>u)=1/u$, $u>1$; other choices such as the unit Fr\'echet distribution are also possible since they lead to the same limits in \eqref{eq:conv-pp} and \eqref{eq:conv-rpp}.  Inference approaches for parametric dependence models can resort to maximization of likelihood variants \citep{TO15,dFD18}, or of other objective functions based on the probability density, such as gradient scores \citep{dFD18, K2024}. The latter approach bypasses the  numerical calculation of the $r$-extremal coefficient $\theta$, which is often challenging when $\theta$ depends on the model parameters. 

We fix a risk functional $r_{\text{data}}:\mathbb{R}_+^D\rightarrow [0,\infty)$ and a threshold $u>0$ to  identify the extreme events, which satisfy $r_{\text{data}}(\mathbf X_i) > u$, and we denote by $n_u$ the number of extreme events and by $i_1,\ldots,i_{n_u}$ their time indices. The realizations of the $r$-Pareto process are given by $\mathbf Z_{i_k} = \mathbf X_{i_k}/u$, $k=1,\ldots,n_u$.

Consider the intensity function  
\begin{equation}\label{eq:intensity-function}
\lambda_{\text{data}}(\mathbf z) = -\frac{\partial^D}{\partial z_1\times \cdots \times z_D} \nu_{\text{data}}([\mathbf{0}, \mathbf{z}]^C)
\end{equation}
of the measure $\nu_{\text{data}}$ defined with respect to $S_{\text{data}}$.
We indicate a parametric model by using its parameter vector $\mathbf{\psi}$ as superscript. The $r$-Pareto log-likelihood function is
$$
\ell(\mathbf{\psi}) = \sum_{i=1}^{n} 1(r(\mathbf{X}_i)>u) \times \log \left(\frac{\lambda_{\text{data}}^{(\mathbf{\psi})}(\mathbf Z_{i})}{\theta^{(\mathbf{\psi})}}\right) = \sum_{k=1}^{n_u} \log \left(\frac{\lambda_{\text{data}}^{(\mathbf{\psi})}(\mathbf Z_{i_k})}{\theta^{(\mathbf{\psi})}}\right),
$$
where $\theta^{(\mathbf{\psi})}=\nu_{\text{data}}\{\mathbf z: r_{\text{data}}(\mathbf z) \geq 1\}$. A computational bottleneck often resides in calculating $\theta^{(\mathbf{\psi})}$ through numerical integration of $\lambda_{\text{data}}^{(\mathbf{\psi})}$, but for some choices of $r_{\text{data}}$ this constant does not depend on $\mathbf{\psi}$, such that fast likelihood computations are possible, \emph{e.g.}, with $\theta=1$ for $r_{\text{data}}(f)=f(s)$ and for $r_{\text{data}}(f)=\sum_{j=1}^D f(s_j)/D$  \citep{Opitz2015,EMKS15,dFD18}. 

\index{partial censoring}

When interest is in  $r_{\text{data}}(f)=\max_{j=1}^D f(s_j)/u_j$ for a threshold vector $(u_1,\ldots,u_D)^T$ and risk threshold $u=1$, partial censoring can be applied, where  components of $\mathbf Z_{i_k}$ that fall below $1$ are considered as censored. Then, partial derivatives in \eqref{eq:intensity-function} are taken only for non-censored components, whereas for censored components $z_j$ is replaced by $u_j$ \citep{TO15,dFD18}. 
Multivariate integrals arising in the log-likelihood function for this approach can be calculated using Quasi-Monte-Carlo techniques \citep{dFD18}. We now provide some likelihood-related formulas pertaining to the Brown--Resnick process.

\begin{example}[Brown--Resnick model]
Various forms of expressions for $\lambda_{\text{data}}$ and related quantities were derived for the Brown--Resnick model \citep{Ribatet2013,Wadsworth2014,EMKS15,dFD18}. As in  \cite{dFD18}, a pivotal role in the following formulas is taken by the first component $s_1$, but the roles of the different locations $s_1,\ldots,s_D$ are indeed exchangeable. First, define a quadratic matrix with $D-1$ rows and columns, 
$$
\Sigma_{\text{data}}^{(\mathbf{\psi})} = \left(\gamma_{i,1}+\gamma_{1,j}-\gamma_{i,j}\right)_{2\leq i,j \leq D}, \quad \gamma_{i,j}=\gamma(s_i,s_j),\ 1\leq i,j\leq D,
$$
where $\gamma$ is the semi-variogram  characterizing the dependence structure. 
With $\tilde{\mathbf{z}} = \log(z_i/z_1)+\gamma_{i,1}$, $j=2,\ldots,D$, we obtain the intensity function
$$
\lambda_{\text{data}}^{(\mathbf{\psi})}(\mathbf z) = \frac{\left|\Sigma_{\text{data}}^{(\mathbf{\psi})}\right|^{-1/2}}{(2\pi)^{(D-1)/2} z_1^2z_2\times \ldots \times z_D }\exp\left(-\frac{1}{2} \tilde{\mathbf{z}}^T \left(\Sigma_{\text{data}}^{(\mathbf{\psi})}\right)^{-1} \tilde{\mathbf z}\right), \quad \mathbf{z} \geq \mathbf{0}.
$$

When we use $r_{\text{data}}(f)=f(s_1)$ for the risk functional, the $r$-extremal coefficient is $\theta=1$ and the probability density of $R$ is $1/r^2$, $r>1$. 
Then, the probability density of the log-spectral process $\mathbf Y=\mathbf Z/R$ (where $R=Z_1$) is 
$
y_1^2\lambda_{\text{data}}^{(\mathbf{\psi})}(\bm y), 
$
and for the Brown--Resnick model this expression corresponds to a log-Gaussian random vector of dimension $D-1$ with mean vector $(-\gamma_{j,1})_{j=2,\ldots,D}$ and covariance matrix $\Sigma_{\text{data}}^{(\mathbf{\psi})}$, as already stated in Equation~\eqref{eq:Y-s_0}. 
\end{example}

\subsubsection{Gradient scoring}
\label{sec:gradient_scoring}

To bypass the calculation of multivariate integrals which becomes prohibitively costly with increasing dimension,  \citet{dFD18} propose using score matching. The purpose of proper scoring rules usually is to compare the predictive performance of various models. However, the minimization of the score between a parametric model and the dataset can also provide a consistent and asymptotically normal estimator under mild conditions \citep{Dawid2016}. The gradient score is based on the gradient of the log-likelihood and provides an interesting choice for estimating (generalized) $r$-Pareto processes since the normalizing constant $\theta$ cancels out. Moreover, a weighting scheme can be applied that mimics partial censoring by down-weighting the contributions of relatively small values in $r(\mathbf Z_{i_k})$ and in the margins of $\mathbf Z_{i_k}$. This approach requires a differentiable risk functional $r$. To overcome this limitation, smooth approximations can be used; for instance, the $p$-norm with large enough $p$ quite accurately approximates the maximum risk functional; see \S\ref{sec:application} and \cite{dFD18}. For detailed formulas of score matching and expressions for the Brown--Resnick model we refer to \cite{dFD18}; an R implementation is available in the package \texttt{mvPot}.

\subsubsection{Likelihood-free neural Bayes estimators}

Efficient likelihood-free inference methods were  recently  developed using neural Bayes estimators, where a neural network is constructed and trained to predict the unknown parameters by using large amounts of realizations of the parametric model, simulated according to many different parameter configurations, as  input data  \citep{RSZH2023,SRZH2023}. Benefits of such estimators  are that they do not require computing normalizing constants in the likelihood and scale very well to settings with very large numbers of data locations in $S_{\text{data}}$. While the training phase of the neural network could still require substantial computational resources, the estimators are amortized, which means that a trained neural estimator can be used at very low computational cost to estimate the parameters for large numbers of datasets.

\section{Application example}
\label{sec:application}

We give some background on the data used for the  application examples below; it is described in detail in the supplementary material. It consists in significant wave heights here denoted by $H_s$, a quantity related to the energy of the waves and, consequently, to their severity. Data is provided by the hindcast sea-states database ResourceCODE \citep{A2021} and is accessible with the R package \texttt{resourcecode} \citep{R2023}. It is provided on a non-regular grid of more than $300000$ nodes covering the European Atlantic waters, more refined close to the coast. At each location, several sea-state parameters are available at hourly time step from 1994 to 2020. In the following, daily data (taken as the observation at 12~am) from 1995 to 2015 are considered, and only the months of January to March are kept, corresponding to the season during which the most extreme $H_s$ are usually observed. We focus on a study area $S$ near the French coast; see Figure~\ref{fig::MapHsData}. In the end, the dataset comprises $1895$ observations of $H_s$ at $100$ locations.

The inference is performed using the \texttt{R} package \texttt{mvPot} \citep{mvPot}, and we follow the accompanying tutorial \citep{mvPotTutorial}. As presented in \S\ref{sec:stat-inference-pareto-processes}, data are first marginally normalized to a standard Pareto distribution: for each location $s\in S$, we estimate a threshold $u(s)$, defined as the $0.95$ empirical quantile at site $s$. Above this thresold, we fit a GP distribution, using for example the function \texttt{fpot} from the \texttt{R} package \texttt{evd}, 
and below we consider the empirical cumulative distribution function. 

We then model the extremal dependence structure of the marginally normalized data by fitting an $r$-Pareto process associated with the risk functional $r(X)=\max_{s\in S}X(s)$, and
we assume a parametric form for the spectral process $Y$. A common choice is to consider a log-Gaussian model, with power variogram $\gamma(h)=(h/\beta)^\alpha$. 

The estimation procedure is performed using the gradient scoring rule approach of \citet{dFD18} described in \S\ref{sec:gradient_scoring}. Complete reproducible \texttt{R} codes are provided in a supplementary file. The main idea is to estimate the parameters by minimizing  the peaks-over-threshold gradient score function. In the \texttt{mvPot} package, this function is named \texttt{scoreEstimation}. It takes as arguments the observed exceedances \texttt{obs},  \textit{i.e.} the events satisfying $r(\bm X_i)>u$; the matrix of coordinates \texttt{loc}; the variogram model \texttt{vario} (defined as a function of the distance $h$); the weighting function \texttt{weightFun} and its derivative \texttt{dWeightFun}; and the marginal threshold $u$,  that is typically taken as a high quantile of $r(\bm X_i)$. As explained in \S\ref{sec:stat-inference-pareto-processes}, POT-stability can be checked to assess the value of $u$. 

\begin{verbatim}
  scoreEstimation(obs, loc, vario, weightFun,dWeightFun, u)
\end{verbatim}

Finally, to validate the estimated dependence model, we consider the spatial extremogram as a measure of extremal dependence: for any two locations $s_1,s_2\in S$, it is defined as the asymptotic conditional exceedance probability
$$\chi(h)=\lim\limits_{u\to\infty}\chi_u(h)=\lim\limits_{u\to\infty}\Pr(X(s_1)>u\vert X(s_2)>u),$$ where $h=\| s_1-s_2\|$. Within the \texttt{mev} package, the empirical extremogram (for fixed threshold $u$) can be obtained using the function $\texttt{extremo}$ by specifying the data matrix \texttt{dat}, the matrix of coordinates \texttt{coord} and the marginal probability threshold \texttt{margp}:
\begin{verbatim}
extremo(dat=hs_dat, margp=0.95, coord=coordinates, plot=T)
\end{verbatim}

\begin{figure}[htb]
\begin{center}\includegraphics[scale=0.4]{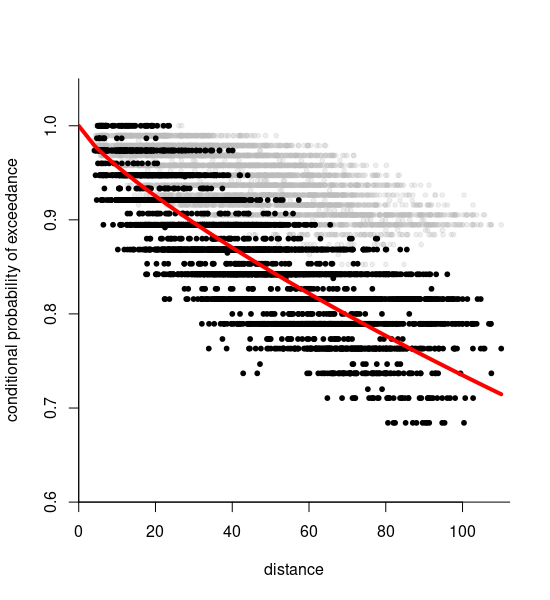}
\end{center}
\caption[Application]{Empirical extremogram $\chi_u$ for $u=0.95$ (light dots) and $u=0.98$ (dark dots), and theoretical extremogram obtained using the estimated parameters (solid line), as functions of the distance.}
\label{fig::app}
\end{figure}

Figure \ref{fig::app} depicts the theoretical values of the extremogram based on the estimated parameters of the dependence model (solid red line), along with the empirical estimates at two different thresholds, as a function of the distance $h$. Looking at the empirical estimates, one can see that for a fixed threshold, even at relatively far distances, the empirical extremogram values remain relatively high and indicate positive extremal dependence. Though,  for higher quantiles, the strength of dependence decreases. This means that a subasymptotic model for spatial extremes, capable to capture dependence strength that decreases relatively strong for more extreme joint events, could be an interesting alternative here.

Finally, given the relatively large amount of data available and potentially nonstationary dependence with respect to distance to the coastline, more sophisticated models than the one illustrated here could be inferred and  prove useful.

\section{Conclusion}

In this chapter, we have presented and illustrated fundamental elements of asymptotic theory, methods and models for generalized $r$-Pareto processes and important subclasses. This approach extends the idea of peaks-over-threshold modeling to multivariate vectors and stochastic processes by allowing for  flexibly choosing a risk functional to characterize joint extreme events as risk exceedances. This spatial modeling framework is still quite recent, and further methodological developments and a more widespread adoption in various applied fields can be expected in the future, especially thanks to its increased flexibility as compared to max-stable processes and its good interpretation in terms of original events and not pointwise maxima. 

A recurrent numerical challenge in peaks-over-threshold modeling is the computation of integrals that arise from censoring mechanisms and are often defined in high dimension equal to the number of data locations. Several methods have proven useful to bypass this bottleneck, such as gradient scoring, still anchored in likelihood-based inference, or completely likelihood-free techniques with recent neural Bayes estimators.

Many case studies, including the application example in this chapter, reveal that POT-stability is not reached in the observed extremes of many environmental processes since the strength of extremal dependence tends to decrease when increasing thresholds. Nevertheless, we emphasize that POT-stable models provide an elegant mathematical framework for modeling joint extreme events and remain useful for inferring extremal dependence characteristics, and also for predicting joint-event probabilities, but predictions will tend to be biased when extrapolating very far into the joint tail.

\bibliographystyle{plainnat}

\end{document}